\documentclass{daj}

\usepackage{amsmath,amssymb,amsthm,bbm,mathtools,comment,cleveref, subfigure}
\usepackage{tikz}
\tikzset{vtx/.style={circle, fill, inner sep=1.5pt}}

\def\eps{\varepsilon}

\def\cF{\mathcal{F}}
\def\cG{\mathcal{G}}
\def\cH{\mathcal{H}}
\def\cS{\mathcal{S}}

\theoremstyle{plain}
\newtheorem*{theorem*}{Theorem}
\newtheorem{theorem}{Theorem}[section]		
\newtheorem{lemma}[theorem]{Lemma}

\newtheorem{conjecture}[theorem]{Conjecture}

\newtheorem{definition}[theorem]{Definition}
\theoremstyle{remark}

\dajAUTHORdetails{%
  title = {Power saving for the Brown-Erd\H os-S\'os problem}, 
  author = {Oliver Janzer, Abhishek Methuku, Aleksa Milojevi\'c, and Benny Sudakov},
  plaintextauthor = {Oliver Janzer, Abhishek Methuku, Aleksa Milojevic, Benny Sudakov},
  keywords = {Brown-Erd\H os-S\'os conjecture},
}   

\dajEDITORdetails{%
   year={2025},
   number={5},
   received={26 November 2023},   
   published={10 July 2025},  
   doi={10.19086/da.138191},       
}   

\begin{document}

\begin{frontmatter}[classification=text]

\title{Power saving for the Brown-Erd\H os-S\'os problem} 

\author[oliver]{Oliver Janzer}
\author[abhishek]{Abhishek Methuku\thanks{Research supported by the UIUC Campus Research Board Award RB25050.}}
\author[aleksa]{Aleksa Milojevi\'c\thanks{Research supported by SNSF grant 200021-228014.}}
\author[benny]{Benny Sudakov\footnotemark[2]}

\begin{abstract}
Let $f(n, v, e)$ denote the maximum number of edges in a 3-uniform hypergraph on $n$ vertices which does not contain $v$ vertices spanning at least $e$ edges. A central problem in extremal combinatorics, famously posed by Brown, Erd\H{o}s and S\'os in 1973, asks whether $f(n, e+3, e)=o(n^2)$ for every $e \ge 3$. A classical result of S\'ark\"ozy and Selkow states that $f(n, e+\lfloor \log_2 e\rfloor+2, e)=o(n^{2})$ for every $e \ge 3$. This result was recently improved by Conlon, Gishboliner, Levanzov and Shapira.

Motivated by applications to other problems, Gowers and Long made the striking conjecture that $f(n, e+4, e)=O(n^{2-\eps})$ for some $\eps=\eps(e)>0$.  Conlon, Gishboliner, Levanzov and Shapira, and later, Shapira and Tyomkyn reiterated the following approximate version of this problem. What is the smallest $d(e)$ for which $f(n, e+d(e), e)=O(n^{2-\eps})$ for some $\eps=\eps(e)>0$? In this paper, we prove that for each $e\geq 3$ we have $f(n, e+\lfloor \log_2 e\rfloor +38, e)=O(n^{2-\eps})$ for some $\eps>0$. This shows that one can already obtain power saving near the S\'ark\"ozy-Selkow bound at the cost of a small additive constant.
\end{abstract}
\end{frontmatter}

\section{Introduction}

Much of extremal combinatorics is concerned with determining which global properties force the appearance of certain local substructures. One of the most important questions of this kind is the \emph{Tur\'an problem}, which asks how many edges an $n$-vertex graph (or hypergraph) can have without containing a given graph (or hypergraph) as a subgraph. The subject of this paper is another very central and closely related problem. A \emph{$(v,e)$-configuration} is a hypergraph with at most $v$ vertices and at least $e$ edges. Estimating the maximum number of edges an $n$-vertex $r$-uniform hypergraph can have without containing a $(v,e)$-configuration is an important topic in extremal graph theory. For example, when $e=\binom{v}{r}$, the problem becomes equivalent to determining the maximum possible number of edges without containing $K_v^{(r)}$, which is a notoriously difficult open problem for $r>2$.

Let us write $f(n, v, e)$ for the maximum number of edges in a $3$-uniform hypergraph on $n$ vertices which does not contain a $(v, e)$-configuration. In 1973, Brown, Erd\H os and S\' os \cite{BES73,BES73_2} initiated the study of this function for various values of $v$ and $e$. The following conjecture, named after them, is one of the most famous open problems in extremal combinatorics.

\begin{conjecture}[Brown--Erd\H os--S\'os conjecture~\cite{BES73,BES73_2}] \label{conj:BESConj}
    For every $e\geq 3$, we have $f(n,e+3,e)=o(n^2)$.
\end{conjecture}

The first result in this direction was the famous $(6, 3)$-theorem shown by Ruzsa and Szemer\'edi \cite{RS76}, who proved that $f(n, 6, 3)=o(n^2)$, thus establishing Conjecture~\ref{conj:BESConj} in the case $e=3$. 
This result has important implications in number theory. For instance, it implies Roth's theorem which states that any dense subset of first $n$ integers contains a three term arithmetic progression.
To this day, $e=3$ is the only instance of the Brown-Erd\H{o}s-S\'os conjecture which has been resolved. 

Therefore, a large amount of work has been put into proving approximate versions of Conjecture~\ref{conj:BESConj}. For instance, a classical result of S\'ark\"ozy and Selkow \cite{SS04} from 2004 is the following.
\begin{theorem}[S\'ark\"ozy and Selkow~\cite{SS04}]
For every $e \ge 3$, we have  $f(n, e+\lfloor \log_2 e\rfloor +2, e)=o(n^2)$.
\end{theorem}
For many years this result has been the state-of-the-art before Solymosi and Solymosi \cite{SS17} showed that $f(n, 14, 10)=o(n^2)$, thus improving the S\'ark\"ozy-Selkow theorem for $e=10$ which only gave $f(n, 15, 10)=o(n^2)$. In a recent breakthrough, Conlon, Gishboliner, Levanzov and Shapira~\cite{CGLS23} extended the ideas of Solymosi and Solymosi significantly to show that $f(n, e+O(\log e/\log\log e), e)=o(n^2)$. 

An important feature of all of the above proofs is their heavy use of the regularity lemma, both in the graph and the hypergraph setting. This means that the bounds they obtain on $f(n, v, e)$ are barely below quadratic. 

Motivated by various applications, Gowers and Long~\cite{GL21} asked to obtain a power-type improvement for the Brown--Erd\H os--S\'os problem.
They had made the surprising conjecture that $f(n, e+4, e)=O(n^{2-\eps})$ for all $e\geq 3$ and some $\eps>0$ potentially depending on $e$. 
They showed that even proving the special case $f(n,9,5)=O(n^{2-\eps})$ of this conjecture would answer a question of Ruzsa~\cite{ruzsa1993solving} on sets of integers avoiding solutions to a certain linear equation.
One cannot expect to strengthen Conjecture~\ref{conj:BESConj} to state $f(n, e+3, e)=O(n^{2-\eps})$ since there are $n$-vertex $3$-uniform hypergraphs with $n^{2-o(1)}$ edges avoiding $(e+3, e)$-configurations, for $e\in \{3, 4, 5, 7, 8\}$. More precisely, Ruzsa and Szemer\'edi~\cite{RS76} constructed a graph with $n$ vertices and $n^{2-o(1)}$ edges avoiding $(6, 3)$-configurations. Moreover, it is not hard to see that every $(7, 4)$- and $(8, 5)$-configuration contains a $(6, 3)$-configuration as a subhypergraph, and hence the Ruzsa-Szemer\'edi construction also shows that $f(n, 7, 4)\geq n^{2-o(1)}$ and $f(n, 8, 5)\geq n^{2-o(1)}$. Recently, lower bounds of the same type for $f(n, 10, 7)$ and $f(n, 11, 8)$ were also found by Ge and Shangguan~\cite{GS21}. Hence, in some sense, the conjecture of Gowers and Long is the strongest possible.

Conlon, Gishboliner, Levanzov and Shapira~\cite{CGLS23}, and later, Shapira and Tyomkyn~\cite{ST23} reiterated the following approximate version of the Gowers-Long conjecture. 
What is the smallest function $d=d(e)$ for which $f(n, e+d, e)=O(n^{2-\eps})$ for some $\eps = \eps(e)>0$? In this paper we prove the following, which can be thought of as a version of the S\'ark\"ozy-Selkow theorem with power saving.

\begin{theorem}\label{thm:power saving}
For every $e \ge 3$, there exists some $\eps>0$ such that $f(n, e+\lfloor \log_2 e\rfloor+38, e)=O(n^{2-\eps})$.
\end{theorem}

We remark that the weaker bound $f(n,e+O(\log_2 e),e)=O(n^{2-\eps})$ was independently proved by different groups of authors. Indeed, a bound of this form was obtained by Conlon~\cite{Conprivate}, by Gishboliner, Levanzov and Shapira \cite{GLSprivate} and by Gao et. al. \cite{Gaoetal}, using different techniques from one another. The purpose of our paper is to show that we obtain power saving already near the S\'ark\"ozy-Selkow bound, at the cost of a small additive constant. Note that in order to obtain power saving in the S\'ark\"ozy-Selkow bound when $e = 3$, an additive constant is necessary because, as remarked earlier, there exist $n$-vertex $3$-uniform hypergraphs with $n^{2-o(1)}$ edges avoiding a $(6, 3)$-configuration.

The value of our small additive constant comes from the fact that while our methods allow us to construct large configurations whose number of edges is not much smaller than the number of vertices, they do not allow us to control the exact number of edges easily. Hence, we have to perform a final cleaning step in which we take edges out of our configuration until exactly $e$ edges remain, and this is where the constant $38$ comes from. 

\subsection{Proof overview and notation}

A natural way to think about the Brown-Erd\H{o}s-S\'os conjecture is in terms of the deficiency. Namely, for a 3-uniform hypergraph $F$, we define the \textit{deficiency} of $F$, denoted by $\Delta(F)$, as the difference between the number of vertices and edges of $F$, i.e. $\Delta(F)=v(F)-e(F)$. Throughout the paper, we will often consider deficiencies of subgraphs of $F$. So for a set of vertices $U\subset V(F)$, we define the deficiency of $U$ as the deficiency of the subgraph of $F$ induced by $U$ and write $\Delta(U)=\Delta(F[U])$. In this language, Conjecture~\ref{conj:BESConj} asks to show that in a 3-uniform hypergraph $\cH$ with $\Omega(n^2)$ edges one can always find a subhypergraph $F\subseteq \cH$ with $e(F)=e$ and $\Delta(F)=3$.

Before we explain how the proof of Theorem~\ref{thm:power saving} works, let us present a toy version of our proof which shows that any hypergraph $\cH$ with $n$ vertices and $\omega(n^{7/4})$ edges contains a $(13, 8)$-configuration. This proof would have two steps - the first would be to find many copies of a given $(8, 4)$-configuration and then to glue two of them together to obtain a $(13, 8)$-configuration. Using standard arguments, we may assume that $\cH$ is linear and tripartite with parts $X, Y,$ and $Z$, of roughly equal size.

We define a bipartite graph between $X$ and $Y$ where $(x, y)\in X\times Y$ is an edge precisely when there is a vertex $z\in Z$ such that $xyz\in E(\cH)$. This graph contains $\omega(n^{7/4})$ edges and thus it contains $\omega(n^3)$ four-cycles. Each of these four-cycles corresponds to a homomorphic copy of the hypergraph $F$ in $\cH$, where $F$ has $8$ vertices $x_1, x_2, y_1, y_2, z_1, z_2, z_3, z_4$ and $4$ edges $x_1y_1z_1, x_2y_1z_2, x_2y_2z_3, x_1y_2z_4$. In fact, it is not hard to see that most copies of $F$ are non-degenerate, thus showing that we have $\omega(n^3)$ actual copies of $F$ in $\cH$. 

We may further randomly partition $X=X_1\cup X_2$, $Y=Y_1\cup Y_2$ and $Z=Z_1\cup Z_2\cup Z_3\cup Z_4$ such that the number of copies of $F$ with $x_i\in X_i,$ $y_i \in Y_i$ for each $1 \le i \le 2$ and $z_i\in Z_i$ for each $1 \le i \le 4$ is still $\omega(n^3)$. Thus, we must have two different copies of $F$ with vertex-sets $\{x_1, \dots, z_4\}$ and $\{x_1', \dots, z_4'\}$, which satisfy $z_1=z_1'$, $z_2=z_2'$ and $z_3=z_3'$ (see Figure~\ref{fig1}). Note that the sets $\{x_1, x_2, y_1, y_2\}$ and $\{x_1', x_2', y_1', y_2'\}$ are disjoint since $\cH$ is linear. Indeed, suppose $x_1=x_1'$ (note that $x_1$ may coincide only with $x_1'$ because of our partition into sets $X_1, \dots, Z_4$).  Then, since $z_1=z_1'$, the linearity of $\cH$ implies $y_1=y_1'$. Similarly, since $z_2=z_2'$, we must also have $x_2=x_2'$. Further, we also get $y_2=y_2'$ and so $z_4=z_4'$, demonstrating that the two copies of $F$ we started with were actually the same copy, a contradiction. An almost identical argument leads to a contradiction if $x_2=x_2'$, $y_1=y_1'$ or $y_2=y_2'$ instead. 

This argument leaves us with only two cases - either there are many copies of $F$ which overlap on the four vertices $z_1, z_2, z_3, z_4$ or there are many pairs of copies of $F$ which overlap only on $z_1, z_2, z_3$, in which case we obtain many non-degenerate copies of our $(13, 8)$-configuration. We will see this dichotomy in the main proof as well, where starting with many copies of the small hypergraph $F$ in $\cH$, we will either be able to find many copies of $F$ glued along a set of deficiency at least $\Delta(F)$ (in which case, by taking the union of many such copies we obtain a hypergraph with many edges but with deficiency still at most $\Delta(F)$) or we will be able to find many pairs of copies of $F$ which are glued in a non-degenerate way along an independent set of size $\Delta(F)-1$. This produces many copies of a new hypergraph $F'$, which allows us to try to repeat the above argument with $F'$ replaced by $F$. The main difficulty lies in controlling the overlaps which may arise, so that we can continue the iteration and obtain the required bound on the deficiency. 

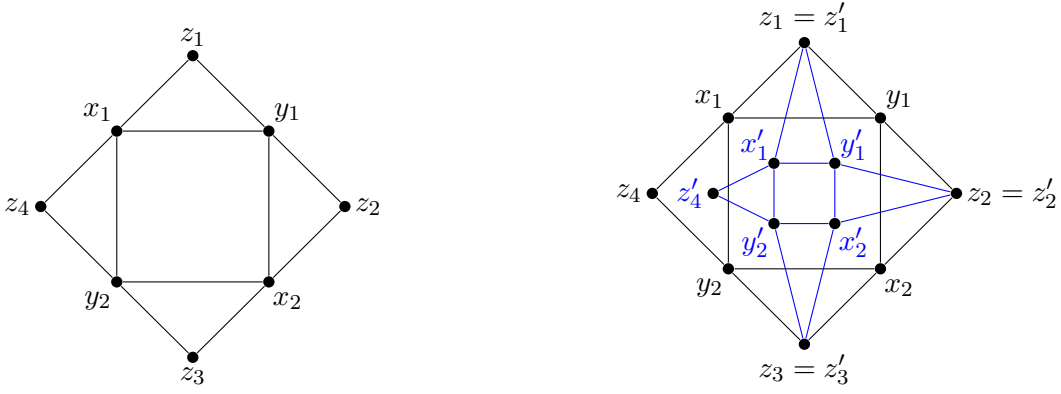
\begin{figure}[t!]
\centering
\begin{subfigure}
  \centering
      \begin{tikzpicture}
        \coordinate[vtx] (x1) at (-1, 1);
        \coordinate[vtx] (y1) at (1, 1);
        \coordinate[vtx] (x2) at (1, -1);
        \coordinate[vtx] (y2) at (-1, -1);

        \coordinate[vtx] (z1) at (0, 2);
        \coordinate[vtx] (z2) at (2, 0);
        \coordinate[vtx] (z3) at (0, -2);
        \coordinate[vtx] (z4) at (-2, 0);
        
        \node at ([shift={(135:0.35)}]x1) {$x_1$}; 
        \node at ([shift={(45:0.35)}]y1) {$y_1$}; 
        \node at ([shift={(-45:0.35)}]x2) {$x_2$}; 
        \node at ([shift={(-135:0.35)}]y2) {$y_2$}; 

        \draw (z1) node[above] {$z_1$};
        \draw (z2) node[right] {$z_2$};
        \draw (z3) node[below] {$z_3$};
        \draw (z4) node[left] {$z_4$};
        
        \draw (x1) -- (y1) -- (x2) -- (y2) -- (x1);
        \draw (x1) -- (z1) -- (y1) -- (z2) --(x2) -- (z3) --(y2) -- (z4) -- (x1);
    \end{tikzpicture}
\end{subfigure}%
~
\begin{subfigure}
  \centering
      \begin{tikzpicture}
        \coordinate[vtx] (x1) at (-1, 1);
        \coordinate[vtx] (y1) at (1, 1);
        \coordinate[vtx] (x2) at (1, -1);
        \coordinate[vtx] (y2) at (-1, -1);

        \coordinate[vtx] (x1') at (-0.4, 0.4);
        \coordinate[vtx] (y1') at (0.4, 0.4);
        \coordinate[vtx] (x2') at (0.4, -0.4);
        \coordinate[vtx] (y2') at (-0.4, -0.4);

        \coordinate[vtx] (z1) at (0, 2);
        \coordinate[vtx] (z2) at (2, 0);
        \coordinate[vtx] (z3) at (0, -2);
        \coordinate[vtx] (z4) at (-2, 0);
        \coordinate[vtx] (z4') at (-1.2, 0);
        
        \node at ([shift={(135:0.35)}]x1) {$x_1$}; 
        \node at ([shift={(45:0.35)}]y1) {$y_1$}; 
        \node at ([shift={(-45:0.35)}]x2) {$x_2$}; 
        \node at ([shift={(-135:0.35)}]y2) {$y_2$}; 
        
        \node[blue] at ([shift={(135:0.35)}]x1') {$x_1'$}; 
        \node[blue] at ([shift={(45:0.35)}]y1') {$y_1'$}; 
        \node[blue] at ([shift={(-45:0.35)}]x2') {$x_2'$}; 
        \node[blue] at ([shift={(-135:0.35)}]y2') {$y_2'$}; 

        \draw (z1) node[above] {$z_1=z_1'$};
        \draw (z2) node[right] {$z_2=z_2'$};
        \draw (z3) node[below] {$z_3=z_3'$};
        \draw (z4) node[left] {$z_4$};
        \draw[blue] (z4') node[left] {$z_4'$};
        
        \draw (x1) -- (y1) -- (x2) -- (y2) -- (x1);
        \draw (x1) -- (z1) -- (y1) -- (z2) --(x2) -- (z3) --(y2) -- (z4) -- (x1);

        \draw[blue] (x1') -- (y1') -- (x2') -- (y2') -- (x1');
        \draw[blue] (x1') -- (z1) -- (y1') -- (z2) --(x2') -- (z3) --(y2') -- (z4') -- (x1');
        
    \end{tikzpicture}
\end{subfigure}
\caption{$(8, 4)$- and $(13, 8)$-configurations}
\label{fig1}
\end{figure}

Let us now briefly sketch how the ideas from this toy proof can be extended to show Theorem~\ref{thm:power saving}. On a very high level, the idea is to glue together smaller hypergraphs without increasing their deficiency too much in order to construct large structures with small deficiency. The first step of our proof will be to construct a sequence of hypergraphs $F_0, F_1, \dots, F_\ell$ such that $e(F_j)= 2^je(F_0)$ and $\Delta(F_j)=\Delta(F_0)+j$, where for every $0 \le j \le \ell-1$, $F_{j+1}$ is obtained by gluing two copies of $F_j$ over a set of deficiency $\Delta(F_j)-1$. More precisely, the hypergraph $F_j$ will have a designated independent set $A_j$ of size $\Delta(F_j)-1$ and the hypergraph $F_{j+1}$ will consist of two copies of $F_j$ which overlap precisely on this independent set $A_j$. Thus, we have $e(F_{j+1})=2e(F_j)$ and $v(F_{j+1})=2v(F_j)-\Delta(F_j)+1$, giving $\Delta(F_{j+1})=\Delta(F_j)+1$.

Then, the hope would be to show by induction that there are at least $n^{\Delta(F_j)-O(\eps)}$ copies of $F_j$ in $\cH$ for every $0 \le j \le \ell$. If this statement holds for the hypergraph $F_j$, then for each of its copies in $\cH$ we mark the set of $\Delta(F_j)-1$ vertices of $\cH$ corresponding to the independent set $A_j\subset V(F_j)$. A typical set of $\Delta(F_j)-1$ vertices of $\cH$ is then marked at least $n^{1-O(\eps)}$ times, so there are at least $n^{\Delta(F_j)-1}(n^{1-O(\eps)})^2$ pairs of copies of $F_j$ overlapping on the set of vertices corresponding to $A_j$. If most of these pairs produce actual copies of $F_{j+1}$, i.e. if they do not have additional overlaps outside the set $A_j$, then $\cH$ contains $\Omega(n^{\Delta(F_j)+1-O(\eps)}) = \Omega(n^{\Delta(F_{j+1})-O(\eps)})$ copies of $F_{j+1}$, which is sufficient to prove our induction step and continue our iteration. However, it is possible that most of these pairs will not give actual copies of $F_{j+1}$ and in this case we will show that there exist many copies of $F_j$ overlapping over the exact same set of vertices $U\subset V(F_j)$ with the property that $A_j\subsetneq U$, thus forming a sunflower-like structure. If we define the independent set $A_j$ carefully, we will be able to show that for any $U\subset V(F_j)$ with $A_j\subset U$ and $A_j \neq U$, we have $\Delta(U)\geq \Delta(F_j)$. In particular, this means that if $A_j\subset U$ and $A_j \neq U$, then every new copy of $F_j$ glued over the set $U$ contributes at least as many edges as it contributes vertices. This means that we can grow our structure arbitrarily without increasing its deficiency, which is precisely what we want.

One feature of our proof is that we are not able to control the configuration that we get exactly. This is because we can only show that $\cH$ either contains hypergraphs $F_j$ (up to a certain value of $j$) or a sunflower-like structure whose petals do not add any deficiency. Since our argument will not necessarily produce a $(v, e)$-configuration for some pre-specified $e \geq 3$, a final stage of cleaning is needed in order to turn the configuration we obtain into a configuration with exactly $e$ edges. In order to do this, we inductively maintain the property that most of the vertices in $F_j$ have degree $1$ (for each $j$). Thus, if we find a configuration with many more edges and vertices than we need, we can remove several vertices of degree $1$ from the configuration without changing its deficiency until the number of edges is exactly $e$.

\textbf{Organisation of the paper.} In Section~\ref{subsec:construction of eligible hypergraphs} we formally define the conditions we require the hypergraphs $F_0, \dots, F_\ell$ to satisfy and show how to construct $F_{j+1}$ from $F_j$ so that all of these properties are maintained. We will also define $F_0$ and show how to find many copies of $F_0$ in any linear hypergraph with $\Omega(n^{2-\eps})$ edges. Then, in Section~\ref{subsec:iteration} we show how to use the fact that $\cH$ contains many copies of $F_j$ to either find many copies of $F_{j+1}$ in $\cH$ or to find a sunflower-like structure. We then use this, along with a procedure which shows how to remove the degree $1$ vertices from a sunflower-like structure to obtain a $(v, e)$-configuration for suitable values of $v$ and $e$, to complete the proof of Theorem~\ref{thm:power saving}.

\section{Proof of Theorem~\ref{thm:power saving}}

\subsection{Construction of eligible hypergraphs}\label{subsec:construction of eligible hypergraphs}

As discussed in the proof overview, our aim is to construct a sequence of hypergraphs $F_0, \dots, F_\ell$ satisfying certain properties. In this section we will state these properties precisely. In particular, we will show how to construct $F_{j+1}$ from $F_j$ while making sure that all the necessary properties are maintained. To be able to state these properties, we need to introduce a special kind of independent set in a hypergraph, which is particularly suited for gluing along it.

\begin{definition}\label{defn:good}
Let $F$ be a 3-uniform hypergraph. We say that an independent set $A\subseteq V(F)$ is \textit{good} if for any set $U\supsetneq A$, we have $\Delta(U)\geq |A|+1$.
\end{definition}

\noindent
Note that the collection of good sets is closed under taking subsets. To see this, suppose that we have a good set $A\subset V(F)$ and $A'\subset A$. To show that $A'$ is good, we need to verify that for any $U\supsetneq A'$ one has $\Delta(U)\geq |A'|+1$. If $U\subset A$, this is clear since $U$ must be an independent set and hence $\Delta(U)=|U|\geq |A'|+1$. Otherwise, by using that $A$ is good and that $A\subsetneq U\cup A$ we obtain $\Delta(U\cup A)\geq |A|+1$. This allows us to show the desired bound on $\Delta(U)$ as follows.
\begin{align*}
    \Delta(U)=v(F[U])-e(F[U])&\geq v(F[U\cup A])-|A\backslash U|-e(F[U\cup A])\\
    &\geq \Delta(U\cup A)-|A\backslash A'|\geq |A|-|A\backslash A'|+1=|A'|+1.
\end{align*}

This discussion allows us to conclude that for any good set $A$ and any set $U\subseteq V(F)$ not fully contained in $A$, the set $A\cap U$ is also a good set and therefore $\Delta(U)\geq |A\cap U|+1$. 

Our definition of good sets vaguely resembles Definition 2.3 in the paper of Conlon, Gishboliner, Levanzov and Shapira \cite{CGLS23}. Namely, they consider independent sets $A$ of size $\Delta(F)+1$ with the property that the deficiency of any set $U\subset V(F)$ is roughly controlled by $|A\cap U|$. However, their definition is more complex than ours, because their gluing procedure (relying on the hypergraph regularity lemma) glues several copies of the hypergraph at a time instead of $2$ as in our paper.

\begin{definition}\label{defn:eligible}
A 3-uniform hypergraph $F$ with $\Delta(F)=k\geq 1$ is \textit{eligible} if it satisfies the following properties:
\begin{itemize}
    \item[(i)] there are two disjoint good sets $A, B\subset V(F)$ of size $|A|=|B|=k-1$,
    \item[(ii)] there exist distinct vertices $u, v \in V(F) \setminus (A \cup B)$,
    \item[(iii)] there are at most $e(F)/4-k$ vertices of $F$ with degree more than 1. Moreover, every edge of $F$ contains at most one vertex of degree $1$, $F$ has no isolated vertices, and
    \item[(iv)] every $U\subseteq V(F)$ of size $|U|\geq 2$ satisfies $\Delta(U)\geq 2$.
\end{itemize}
\end{definition}

\noindent
Note that all eligible hypergraphs have $e(F)\geq 4k$, by (iii). The following lemma shows how to construct the hypergraph $F_{j+1}$ based on $F_j$ and shows that the eligibility is maintained under this construction. The hypergraphs $F_j$ constructed by repeated applications of this lemma will have deficiency $\Delta(F_j)=j$ and $e(F_j)\approx 2^j$ edges.

\begin{lemma}\label{lemma:maintaining eligibility}
Let $F$ be an eligible hypergraph with $\Delta(F)=k$. Suppose $F'$ is the hypergraph obtained by taking two vertex-disjoint copies of $F$ and only identifying the vertices of the sets $A$ in both copies. Then $F'$ is an eligible hypergraph with $\Delta(F')=k+1$ and $e(F')=2e(F)$.
\end{lemma}
\begin{proof}
Let us denote the two copies of $F$ by $F_1$ and $F_2$. Since $A$ is an independent set, no edges of $F_1$ and $F_2$ are identified and hence we have $e(F')=2e(F)$. Moreover, since exactly $k-1$ vertices are identified, we have $v(F')=2v(F)-(k-1)$ and therefore \[\Delta(F')=v(F')-e(F')=2v(F)-(k-1)-2e(F)=2k-(k-1)=k+1.\]

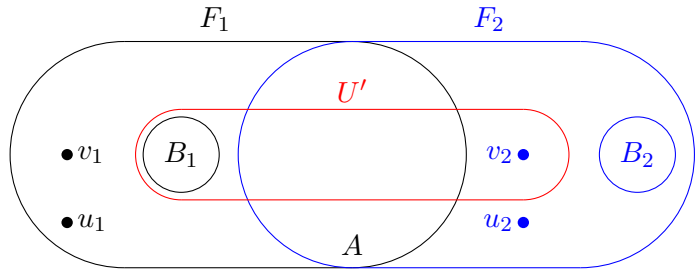
\begin{figure}[h!]
    \centering
    \begin{tikzpicture}
  \def\R{1.5}

  \draw (-2*\R, \R) arc(90:270:\R);
  \draw (-2*\R, \R) -- (0, \R);
  \draw (-2*\R, -\R) -- (0, -\R);
  \draw (0, \R) arc(90:-90:\R);
  \draw[blue] (2*\R, \R) arc(90:-90:\R);
  \draw[blue] (0, \R) arc(90:270:\R);
  \draw[blue] (2*\R, \R) -- (0, \R);
  \draw[blue] (2*\R, -\R) -- (0, -\R);
  
  \draw (-1.5*\R, 0) circle(0.5);
  \coordinate[vtx] (v1) at (-2.5*\R, 0);
  \coordinate[vtx] (u1) at (-2.5*\R, -0.9);

  \draw[blue] (2.5*\R, 0) circle(0.5);
  \coordinate[vtx, blue] (v2) at (1.5*\R, 0);
  \coordinate[vtx, blue] (u2) at (1.5*\R, -0.9);

  \draw[red] (-1.5*\R, 0.6) arc(90:270:0.6);
  \draw[red] (-1.5*\R, 0.6) -- (1.5*\R, 0.6);
  \draw[red] (-1.5*\R, -0.6) -- (1.5*\R, -0.6);
  \draw[red] (1.5*\R, 0.6) arc(90:-90:0.6);
  \node at (-1.2*\R, 1.2*\R) {$F_1$}; 
  \node[blue] at (1.2*\R, 1.2*\R) {$F_2$}; 
  \node[right] at (v1) {$v_1$};
  \node[right] at (u1) {$u_1$};
  \node at (-1.5*\R, 0) {$B_1$};
  \node[left, blue] at (v2) {$v_2$}; 
  \node[left, blue] at (u2) {$u_2$}; 
  \node[blue] at (2.5*\R, 0) {$B_2$};
  \node[red] at (0, 0.85) {$U'$}; 
  \node at (0, -0.8*\R) {$A$}; 
    \end{tikzpicture}
    \caption{Illustration of the proof of Lemma~\ref{lemma:maintaining eligibility}.}
\end{figure}

It remains to show that $F'$ is an eligible hypergraph. We begin by introducing some notation and defining the sets $A', B'$ and vertices $u', v'$ satisfying the conditions $\mathrm{(i)}$ and $\mathrm{(ii)}$ of Definition~\ref{defn:eligible}.

For $i\in \{1,2\}$, since $F_i$ is eligible, it contains a good set $B_i$ disjoint from $A$, and additional vertices $u_i, v_i \in V(F_i) \setminus (A\cup B_i)$. To show that $F'$ is eligible we define two sets $A' \coloneqq B_1\cup \{v_2\}$, $B'\coloneqq B_2\cup \{v_1\}$, along with additional vertices $u' \coloneqq u_1$ and $v' \coloneqq u_2$. Note that $u'$ and $v'$ are distinct and $u', v' \in V(F') \setminus (A' \cup B')$, verifying the second condition in Definition \ref{defn:eligible}. Finally, for every set $U'\subset V(F')$, we define the sets $U^{(1)}=U'\cap V(F_1)$ and $U^{(2)}=U'\cap V(F_2)$.

Let us now show that $A', B'\subset V(F')$ are good sets with respect to $F'$. Since the definitions of $A'$ and $B'$ are completely symmetric, it suffices to verify the conditions for only one of these sets, say $A'$. Note that $|A'|=|B_1|+1=k=\Delta(F')-1$, which shows that $A'$ has the required size. Further, $A'$ is an independent set since there are no edges between $F_1\backslash A$ and $F_2\backslash A$, and so, in particular, there are no edges between $B_1$ and $v_2$. To show that $A'$ is a good set we need to show that for any set $U'$ with $A'\subsetneq U'\subseteq V(F')$ one has $\Delta(U')\geq |A'|+1=k+1$. We will establish this inequality by using the properties of good sets $B_1\subseteq V(F_1)$ and $A\subseteq V(F_1)$. We have two cases: either $U^{(1)}\supsetneq B_1$ or $U^{(1)}=B_1$ (since $A'=B_1\cup \{v_2\}\subset U'$, we must have $B_1\subseteq U^{(1)}$). 

In the case that $U^{(1)}\supsetneq B_1$, we use the assumption that $B_1$ is a good set in $F_1$ to conclude $\Delta(U^{(1)})\geq |B_1|+1$. Furthermore, we know that $A\cap U^{(2)}$ is a good set in $F_2$ and $A\cap U^{(2)}\subsetneq U^{(2)}$ since $v_2\in U^{(2)}$. Thus, we have $\Delta(U^{(2)})\geq |A\cap U^{(2)}|+1 = |A\cap U'|+1$. These two bounds on $\Delta(U^{(1)})$ and $\Delta(U^{(2)})$ are sufficient to bound the deficiency of $U'$ in $F'$ as follows.
\begin{align*}
    \Delta(U')&=v(F'[U'])-e(F'[U']) \\ &=v(F_1[U^{(1)}])+v(F_2[U^{(2)}])-|U^{(1)}\cap U^{(2)}|-e(F_1[U^{(1)}])-e(F_2[U^{(2)}])\\
    &=\Delta(U^{(1)})+\Delta(U^{(2)})-|A\cap U'|\geq|B_1|+1+|A\cap U'|+1-|A\cap U'|=|B_1|+2=k+1.
\end{align*}

In the case that $U^{(1)}=B_1$, there are no vertices in $U^{(1)}\cap U^{(2)}=U^{(1)}\cap A$ and there are at least two vertices in $U^{(2)}$ since $B_1\cup \{v_2\}\subsetneq U'$. Hence, $\Delta(U^{(1)})=|B_1|$ since $B_1$ is an independent set, and $\Delta(U^{(2)})\geq 2$ since $F_2$ satisfies condition $\mathrm{(iv)}$ of Definition~\ref{defn:eligible}. Hence,
\begin{align*} \Delta(U') &=v(F_1[U^{(1)}])+v(F_2[U^{(2)}])-e(F_1[U^{(1)}])-e(F_2[U^{(2)}])\\
&=\Delta(U^{(1)})+\Delta(U^{(2)})\geq |B_1|+2=k+1.
\end{align*}

This completes the proof that $A'$ is a good set, thus verifying the first condition (of Definition~\ref{defn:eligible}) for $F'$ to be eligible. Further, to show that $F'$ satisfies the third condition, note that any vertex of $F'$ having degree more than 1 is either in $A$ or it already had degree more than 1 in $F_1$ or $F_2$. Hence, there are at most $2(e(F)/4-k)+k-1=e(F')/4-k-1$ vertices of degree more than 1 in $F'$. Moreover, it is clear that every edge of $F'$ has at most one vertex of degree $1$, since any vertex with degree $1$ in $F'$ must also have degree $1$ in $F_1$ or $F_2$. It is also clear that $F'$ has no isolated vertex, since $F_1$ and $F_2$ do not have an isolated vertex.

Finally, to show that $F'$ satisfies the fourth condition, we now verify that every set $U'\subseteq V(F')$ of size $|U'|\geq 2$ satisfies $\Delta(U')\geq 2$. Suppose for a contradiction that there is a set $U'\subseteq V(F')$ of size $|U'|\geq 2$ such that $\Delta(U')\leq 1$. Then, the assumption that $F_1$ and $F_2$ are eligible implies that we cannot have $U^{(1)}=U'$ or $U^{(2)}=U'$. This means that both $U^{(1)}$ and $U^{(2)}$ contain vertices outside $A$. Since $A\cap U'$ is a good set for both $F_1$ and $F_2$, we know that $\Delta(U^{(1)})\geq |A\cap U'|+1$ and $\Delta(U^{(2)})\geq |A\cap U'|+1$. Thus, 
\begin{align*}
    \Delta(U')&=v(F_1[U^{(1)}])+v(F_2[U^{(2)}])-|U^{(1)}\cap U^{(2)}|-e(F_1[U^{(1)}])-e(F_2[U^{(2)}])\\
    &=\Delta(U^{(1)})+\Delta(U^{(2)})-|A\cap U'|\geq |A\cap U'|+2\geq 2,
\end{align*}
completing the proof of the lemma.
\end{proof}

\noindent
To start the iteration, we need to show how to find eligible graphs. The following lemma shows how to construct an eligible hypergraph $F_0$ and find many copies of it in $\cH$.

\begin{lemma}\label{lemma:start}
Let $1\leq s\leq t$ be positive integers. Let $K_{s, t}^+$ be the hypergraph obtained by adding a single vertex $v_e$ to every edge $e\in E(K_{s, t})$ of the complete bipartite graph $K_{s, t}$. For any $0<\eps<1/s^3t$ and sufficiently large $n$, every linear hypergraph $\cH$ with $n$ vertices and $\Omega(n^{2-\eps})$ edges contains $\Omega(n^{s+t-s^3t\eps})$ copies of $K_{s, t}^+$. Moreover, if $8(s+t)\leq st$, the hypergraph $K_{s, t}^+$ is eligible and satisfies $\Delta(K_{s, t}^+)=s+t$ and $e(K_{s, t}^+)=st$.
\end{lemma}
\begin{proof}
To show that every linear hypergraph $\cH$ with $\Omega(n^{2-\eps})$ edges contains many copies of $K_{s, t}^+$, we take a tripartite subgraph $\cH'\subseteq \cH$ containing $\Omega(e(\cH))$ edges and whose parts all have size at least $n/4$ (which can be done by randomly partitioning the vertex set). Let us denote the three parts by $X, Y, Z$, and consider an auxiliary colored graph $G$ on the vertex set $X\cup Y$. Vertices $x\in X, y\in Y$ are adjacent in $G$ precisely when there exists a vertex $z\in Z$ such that $xyz$ is an edge of $\cH$. In this case, the edge $xy$ receives color $z$. Since $\cH$ is linear, each edge of $G$ is assigned a unique color. Furthermore, the coloring is proper for the same reason. Finally, $G$ contains $\Omega(n^{2-\eps})$ edges and $\Theta(n)$ vertices.

The reason behind defining the graph $G$ is the simple observation that every rainbow copy of $K_{s, t}$ in $G$ corresponds to a copy of $K_{s, t}^+$ in $\cH$. Hence, it suffices to show that there are $\Omega(n^{s+t-s^3t\eps})$ rainbow copies of $K_{s, t}$ in $G$. To do this, we use a simple lemma from the paper of Keevash, Mubayi, Sudakov and Verstra\"ete \cite{KMSV07} which states that in any properly colored copy of $K_{s, t'}$, for $t'>(s(s-1)+1)(t-1)$ one can find a rainbow copy of $K_{s, t}$. Let us fix $t'=s^2t$ and use the supersaturation of $K_{s, t'}$ in $G$ to obtain at least $\Omega(n^{s+t'-st'\eps})$ copies of $K_{s, t'}$ in $G$ (see e.g., the K\H{o}v\'{a}ri-S\'{o}s-Tur\'{a}n theorem~\cite{kHovari1954problem}). Then, using Lemma 2.3 from \cite{KMSV07}, in each of these copies one can find a rainbow copy of $K_{s, t}$. Since each copy of $K_{s, t}$ is contained in at most $n^{t'-t}$ copies of $K_{s, t'}$, we conclude that we have at least $\Omega(\frac{n^{s+t'-st'\eps}}{n^{t'-t}})=\Omega(n^{s+t-st'\eps})=\Omega(n^{s+t-s^3t\eps})$ rainbow copies of $K_{s, t}$, as desired.

The number of edges of $K_{s, t}^+$ is equal to the number of edges of $K_{s, t}$ and so $e(K_{s, t}^+)=st$. Further, since we add a new vertex to $K_{s, t}$ for every edge, the total number of vertices of $K_{s, t}^+$ is $v(K_{s, t}^+)=s+t+st$ and so $\Delta(K_{s, t}^+)=s+t$.

Finally, if $8(s+t)\leq st$, then we will show that $K_{s, t}^+$ is eligible. Note that if $T$ is a spanning tree of $K_{s, t}$, then the set $A_T=\{v_e | e\in T\}$ is a good set of size $s+t-1=\Delta(K_{s, t}^+)-1$. To verify this observation, let us take an arbitrary set $U\supsetneq A_T$ and show that $\Delta(U)\geq |A_T|+1$. First, observe that the statement is trivial if $U\cap V(K_{s,t})=\emptyset$ since in this case $U$ is an independent set and $\Delta(U)=|U|$. Hence, assume that $U\cap V(K_{s,t})\neq \emptyset$. Now note that it suffices to show $\Delta(U)\geq |A_T|+1$ only for sets $U$ completely contained within $A_T\cup V(K_{s, t})$. Once the statement is shown for such sets $U$, adding vertices outside $A_T\cup V(K_{s, t})$, which all have degree exactly 1 in $K_{s, t}^+$, cannot decrease the deficiency of $U$. Hence, we focus on sets $U$ with $A_T\subsetneq U\subseteq A_T\cup V(K_{s, t})$. The number of edges of $K_{s, t}^+$ spanned by such a set $U$ is precisely the number of edges of $T$ spanned by the intersection $U\cap V(K_{s, t})$, i.e. $e(K_{s, t}^+[U])=e(T[U\cap V(K_{s, t})])$. Since $T$ is a tree, we must have $e(T[U\cap V(K_{s, t})])\leq |U\cap V(K_{s, t})|-1$ and so 
\[\Delta(U)=|A_T|+|U\cap V(K_{s, t})|-e(K_{s, t}^+[U])\geq |A_T|+|U\cap V(K_{s, t})|-|U\cap V(K_{s, t})|+1=|A_T|+1.\]

Hence, to construct two disjoint good sets in $K_{s, t}^+$ of size $\Delta(K_{s, t}^+)-1=s+t-1$, it suffices to take two edge-disjoint spanning trees of $K_{s, t}$, which is certainly possible as the condition $st\geq 8(s+t)$ implies $s, t\geq 9$ (verifying (i) of Definition~\ref{defn:eligible}). Of course, $K_{s, t}^+$ contains two distinct additional vertices besides these good sets (verifying (ii) of Definition~\ref{defn:eligible}). Further, the vertices of degree more than $1$ in $K_{s, t}^+$ are precisely the initial vertices of $K_{s, t}$ and we have at most $s+t\leq st/4-\Delta(K_{s, t}^+)$ such vertices by the assumption that $st\geq 8(s+t)$. Moreover, every edge of $K_{s, t}^+$ contains exactly one vertex of degree 1 and $K_{s,t}^+$ has no isolated vertex (verifying (iii) of Definition~\ref{defn:eligible}). Finally, since every edge of $K_{s, t}^+$ contains a distinct vertex of degree $1$ and two vertices of $V(K_{s, t})$, any set of $e$ edges of $K_{s, t}^+$ must contain at least $e+2$ distinct vertices. This shows that there are no sets $U\subset V(K_{s, t}^+)$ of deficiency $\Delta(U)\leq 1$ and size $|U|\geq 2$ (verifying (iv) of Definition~\ref{defn:eligible}). Hence, $K_{s, t}^+$ is eligible.
\end{proof}

\subsection{Iteration}\label{subsec:iteration}

We begin this section by showing that if we have many copies of a smaller hypergraph $F_j$ in $\cH$, then we can either glue them together to produce many copies of a larger hypergraph $F_{j+1}$ or  we find an alternative structure in $\cH$ (as mentioned in the proof overview). We now define this structure.

\begin{definition}
Let $F$ be a hypergraph and let $r$ be a positive integer. A hypergraph $\tilde F$ is an \textit{$(r, F)$-sunflower} if there exists a set $U\subsetneq V(F)$ with $\Delta(U)\geq \Delta(F)$ and $r$ embeddings $\varphi_1, \dots, \varphi_r:V(F)\to V(\tilde F)$ satisfying the following two conditions:
\begin{itemize}
    \item[(i)] the embeddings $\varphi_1, \dots, \varphi_r$ cover $\tilde F$, i.e. $\bigcup_{i=1}^r \varphi_i(V(F))=V(\tilde F)$, and 
    \item[(ii)] for different indices $1\leq i<j\leq r$ we have $\varphi_i(u)=\varphi_j(v)$ if and only if $u=v\in U$. In particular, the sets $\varphi_i(V(F)\backslash U)$ and $\varphi_j(V(F)\backslash U)$ are disjoint when $i\neq j$.
\end{itemize}
\end{definition}

The main reason sunflowers are useful for us is that the deficiency of $\tilde F$ is bounded by the deficiency of $F$. Indeed, $v(\tilde F)=|U|+r(v(F)-|U|)$ and $e(\tilde F)=e(F[U])+r(e(F)-e(F[U]))$, so the deficiency of $\tilde F$ can be bounded as follows. $\Delta(\tilde F)=v(\tilde F)-e(\tilde F)=\Delta(F)+(r-1)(\Delta(F)-\Delta(U))\leq \Delta(F)$.

The following lemma, whose proof uses ideas from the proof of Lemma 3.1 in \cite{CGLS23}, shows that any hypergraph $\cH$ with many copies of $F$ contains either an $(r, F)$-sunflower or many copies of $F'$ (obtained by gluing copies of $F$). In the proof of Theorem~\ref{thm:power saving} we will mostly use this lemma with $m=2$.

\begin{lemma}\label{lemma:iteration step}
Let $r, m\geq 2$ be fixed integers and let $F$ be an eligible hypergraph with deficiency $\Delta(F)=k$. Then, there exists a hypergraph $F'$ satisfying $\Delta(F')=k+m-1$, $e(F')=m\cdot  e(F)$ and a positive constant $n_0=n_0(r,m,F)$ with the following property: for all $\eps\leq 1/2$ and all hypergraphs $\cH$ on $n>n_0$ vertices containing $\Omega(n^{k-\eps})$ copies of $F$, either
\begin{itemize}
    \item $\cH$ contains an $(r, F)$-sunflower, or
    \item $\cH$ contains $\Omega(n^{k+m-1-m\eps})$ copies of $F'$.
\end{itemize}
Moreover, if $m=2$, the resulting graph $F'$ is also eligible.
\end{lemma}
\begin{proof}
We take $F'$ to be the hypergraph obtained by gluing $m$ copies of $F$ along a good set $A\subset V(F)$ of size $k-1$. In other words, to obtain $F$ we start from $m$ disjoint copies $F^{(1)}, \dots, F^{(m)}$ and for each vertex $v\in A$ we identify the $m$ instances of that vertex with each other. For $m=2$, this construction is identical to the construction from Lemma~\ref{lemma:maintaining eligibility} and therefore the resulting hypergraph $F'$ is indeed eligible. 

We now focus on proving the main statement of the lemma. Throughout the proof, we identify the vertex set of $F$ with $[v(F)]$, where the set $A \subset V(F)$ is identified with $[k-1]$, and we think of copies of $F$ in $\cH$ as embeddings $\varphi:V(F)\to V(\cH)$. As a first step, we randomly partition vertices of $\cH$ into $v(F)$ parts $S_1, \dots, S_{v(F)}$. We call a copy of $F$ \textit{proper} if the corresponding embedding $\varphi:[v(F)]\to V(\cH)$ satisfies $\varphi(i)\in S_i$ for all $i\in V(F)$. The expected number of proper copies of $F$ is $v(F)^{-v(F)}$ times the total number of copies of $F$ in $\cH$. Hence, we can fix a partition $S_1, \dots, S_{v(F)}$ with $\Omega(n^{k-\eps})$ proper copies of $F$. Let us denote the collection of these copies by $\cF$. For two embeddings $\varphi_1, \varphi_2\in \cF$, we define their intersection $U(\varphi_1, \varphi_2)=\{i\in V(F):\varphi_1(i)=\varphi_2(i)\}$. The random partitioning we performed ensures that any overlap between two copies $\varphi_1, \varphi_2\in \cF$ can only happen on the vertices playing the same role in both copies, i.e. we can never have $\varphi_1(i)=\varphi_2(j)$ for $i\neq j$. Formally, for any $\varphi_1, \varphi_2\in \cF$ we have $\varphi_1(V(F))\cap \varphi_2(V(F))=\varphi_1(U(\varphi_1, \varphi_2))$.

We now define an auxiliary graph $\cG$ whose set of vertices is $\cF$. We define two copies $\varphi_1, \varphi_2\in \cF$ to be adjacent in $\cG$ if $A\subsetneq U(\varphi_1, \varphi_2)$. We now have two cases - either the auxiliary graph $\cG$ has a vertex of degree at least $v(F)!(2r)^{v(F)}$ or all vertices of $\cG$ have degree less than $v(F)!(2r)^{v(F)}$.

Suppose we are in the first case and we have an embedding $\varphi_0:V(F)\to V(\cH)$ with at least $v(F)!(2r)^{v(F)}$ neighbours in $\cG$. Then, by the pigeonhole principle one can find a set $U_0$ such that at least $v(F)!r^{v(F)}$ embeddings $\varphi\in N_{\cG}(\varphi_0)$ satisfy $U(\varphi, \varphi_0)=U_0$. We denote the set of these embeddings by $\cF_0$. Let us now recall the Erd\H{o}s-Rado sunflower lemma, which we plan to apply to the vertex sets of the copies of $F$ in $\cF_0$.

\begin{theorem*} [Erd\H{o}s-Rado sunflower lemma, \cite{ER60}]

Let $\cS$ be a family of sets, where each set has cardinality $k$. If $|\cS|\geq k!(r-1)^k$, there exist $r$ distinct sets $S_1, \dots, S_r\in \cS$ which satisfy $S_i\cap S_j=S_1\cap \dots\cap S_r$ for any $1\leq i< j\leq r$.
\end{theorem*}

\noindent
We apply the Erd\H{o}s-Rado sunflower lemma to the family $\cS=\{\varphi(V(F))|\varphi\in \cF_0\}$. Note that each set in this family has cardinality $v(F)$. Since $|\cS|\geq v(F)!r^{v(F)}$, we obtain a set of $r$ different embeddings $\varphi_1, \dots, \varphi_r$ and a set $U\subsetneq V(F)$ such that $U(\varphi_i, \varphi_{j})=U$ for all $i, j\in [r]$, $i\neq j$. We claim that the union of these $r$ copies of $F$ forms an $(r, F)$-sunflower contained in $\cH$. Since $U(\varphi_0, \varphi_i)\cap U(\varphi_0, \varphi_{j})\subseteq U(\varphi_i, \varphi_{j})$ for all $i, j\in [r]$, we see that $U_0\subseteq U$. By the definition of the adjacency relation in $\cG$, we have $A\subsetneq U_0\subset U$ and therefore the set $U$ has deficiency $\Delta(U)\geq |A|+1=\Delta(F)$. Hence,  $\bigcup_{i=1}^r \varphi_i(F)$ really is an $(r, F)$-sunflower.

Now, we consider the second case, in which all vertices of $\cG$ have degree less than $v(F)!(2r)^{v(F)}$. In this case, $\cG$ contains an independent set $I$ of size at least $\frac{|V(\cG)|}{v(F)!(2r)^{v(F)}}=\Omega(n^{k-\eps})$. For each $(k-1)$-tuple $\mathbf{v}=(v_1, \dots, v_{k-1})\in S_1\times\dots\times S_{k-1}$ representing the vertices of $A$, we define $\cF_{\mathbf{v}}$ to be the set of embeddings $\varphi\in I$ having $\varphi(i)=v_i$ for all $i\in [k-1]$. Note that any two embeddings $\varphi_1, \varphi_2\in \cF_{\mathbf{v}}$ have $U(\varphi_1, \varphi_2)=A=[k-1]$ since they are not adjacent in the graph $\cG$. Therefore, for any $m$ distinct embeddings $\varphi_1, \varphi_2, \ldots, \varphi_m \in \cF_{\mathbf{v}}$, the subgraph $\varphi_1(F)\cup \ldots \cup \varphi_m(F)$ of $\cH$ is a copy of $F'$. In particular, this shows that there are $\Omega\left(\binom{|\cF_{\mathbf{v}}|}{m}\right)$ copies of $F'$ containing the $(k-1)$-tuple $\mathbf{v}$. Since the copies of $F'$ corresponding to different $(k-1)$-tuples $\mathbf{v}$ are different, the number of copies of $F'$ in $\cH$ is at least $\sum_{\substack{\mathbf{v}\in S_1\times\cdots\times S_{k-1}\\|\cF_{\mathbf{v}}|\geq m}} \binom{|\cF_{\mathbf{v}}|}{m}$. Since $x\mapsto \binom{x}{m}$ is a convex function when $x\geq m$, applying Jensen's inequality gives
\begin{align*}
    \sum_{\substack{\mathbf{v}\in S_1\times\cdots\times S_{k-1}\\|\cF_{\mathbf{v}}|\geq m}} \binom{|\cF_{\mathbf{v}}|}{m}&\geq \prod_{i=1}^{k-1}|S_i|\binom{(\prod_{i=1}^{k-1}|S_i|)^{-1}\sum_{|\cF_{\mathbf{v}}|\geq m} |\cF_{\mathbf{v}}|}{m} \\&\geq \frac{\left(\sum_{\mathbf{v}} |\cF_{\mathbf{v}}|-mn^{k-1}\right)^m}{m^m (\prod_{i=1}^{k-1}|S_i|)^{m-1}}
    \geq \Omega\left(\frac{n^{mk-m\eps}}{n^{(m-1)(k-1)}}\right)= \Omega\left(n^{m+k-1-m\eps}\right),
\end{align*}
where the last inequality comes from the fact that the sum $\sum_{\mathbf{v}} |\cF_{\mathbf{v}}|$ simply counts the number of proper copies of $F$ in $\cH$ and that each of the sets $S_1, S_2, \ldots, S_k$ has size at most $n$. Therefore, $\cH$ contains $\Omega(n^{k+m-1-m\eps})$ copies of $F'$, as needed. This completes the proof of the lemma.
\end{proof}

\begin{lemma}\label{lemma:sunflower}
Suppose that $F$ is an eligible hypergraph and suppose that a hypergraph $\cH$ contains an $(r, F)$-sunflower (for some $r$) with at least $e$ edges. If $e(F)\leq e/2$, then $\cH$ contains an $(e+\Delta(F), e)$-configuration.
\end{lemma}
\begin{proof}
By our assumption, $\cH$ contains an $(r, F)$-sunflower with at least $e$ edges. Since it is a sunflower, it has deficiency at most $\Delta(F)$. However, this sunflower may have too many edges and vertices, so it does not immediately give us an $(e+\Delta(F), e)$ configuration. Thus, the idea is to remove some vertices of degree $1$ from the copies of $F$ forming this sunflower and thus reduce the number of edges and vertices of the sunflower, without changing the deficiency in the process.

Suppose the copies of $F$ forming this sunflower correspond to embeddings $\varphi_1, \dots, \varphi_r$ and suppose that these copies are glued over a set $U\subset V(F)$ with $\Delta(U)\geq \Delta(F)$. We call the set $U$ the \textit{core} and  for each copy of $F$ forming this sunflower, we call the set $P=V(F)\backslash U$ a \textit{petal}. Further, let $e_U=e(F[U])$ denote the number of edges of $F$ completely contained in the core and let $e_P=e(F)-e(F[U])$ denote the number of edges incident to a vertex in the petal $P$. The assumption that $\Delta(U)\geq \Delta(F)$ implies that $e_P\geq |P|$, which means that every petal adds at least as many edges as vertices. Note that the sunflower constructed using only the copies $\varphi_1, \dots, \varphi_p$ would have $p|P|+|U|$ vertices and $pe_P+e_U$ edges. Let us fix $p$ to be the smallest number of petals one needs to attach to the core to get at least $e$ edges in the sunflower, i.e. $p=\lceil\frac{e-e_U}{e_P}\rceil$. Note that we may assume that $p\geq 3$ since $p=2$ implies that the sunflower has exactly $e$ edges (because we assumed $e(F) \le e/2$) and hence it represents an $( e+\Delta(F), e)$-configuration.

The total number of vertices in this sunflower is $p|P|+|U|$ and thus if $p|P|+|U|\leq e+\Delta(F)$, we are done since we have found an $(e+\Delta(F), e)$-configuration. Otherwise, $p|P|+|U|> e+\Delta(F)$, in which case our goal is to remove some vertices of degree 1 to obtain exactly $e+\Delta(F)$ vertices. Then, since the deficiency is still at most $\Delta(F)$, this gives us the desired $(e+\Delta(F), e)$-configuration. We claim that it is always sufficient to remove at most $|P|$ vertices of degree $1$ in order to bring the number of vertices to $e+\Delta(F)$. Indeed, note that although the number of vertices of the sunflower formed by $\varphi_1, \dots, \varphi_p$ is larger than $e+\Delta(F)$, it cannot be much larger. Namely, by the minimality of $p$ we have $(p-1) e_P+e_U\leq e$ and so 
\begin{align*}
    p|P|+|U|&= (p-1)|P|+ v(F)\leq (p-2)e_P + |P| + e(F)+\Delta(F)\\
    &= (p-1)e_P+e_U+|P|+\Delta(F)\leq e+\Delta(F)+|P|.
\end{align*}

Thus, it suffices to show that the sunflower consisting of $\varphi_1, \dots, \varphi_p$ has at least $|P|$ vertices of degree $1$. Let $v_U$ denote the number of vertices $v$ of $U$ such that $v$ has degree $1$ in $F$ and the edge incident to $v$ does not contain any vertices of $P$, and let $v_P$ be the number of vertices of $P$ which have degree $1$. Then the number of vertices of degree $1$ in the sunflower consisting of $\varphi_1,\dots,\varphi_p$ is equal to $v_U+pv_P$. Our goal is to show that $v_U+pv_P\geq |P|$. Suppose this is not the case and we have $v_U+pv_P<|P|$. Since $F$ is eligible, at most $e(F)/4$ vertices in the hypergraph $F$ can have degree more than $1$, so we have $v_P\geq |P|-e(F)/4$. We know that $p\geq 3$, so $v_U <  |P|-3v_P$.

Since $F$ is eligible, at least $3e(F)/4$ vertices in $F$ have degree $1$ and moreover, every edge of $F$ contains at most one vertex of degree $1$. So there are at least $3e(F)/4$ edges in $F$ incident to a vertex of degree $1$. Out of these, only $v_U$ are completely contained in $U$. Thus, the number of edges incident to $P$ is at least \[e_P\geq 3e(F)/4-v_U\geq 3e(F)/4-|P|+3v_P\geq 3e(F)/4-|P|+3(|P|-e(F)/4)=2|P|.\] 

In other words, every petal adds twice as many edges as vertices. We will show that this contradicts our assumption that $p|P| + |U| > e + \Delta(F)$. More precisely, we have 
\begin{align*}
p|P|+|U|\leq &\left(\frac{e-e_U}{e_P}+1\right)|P| +|U|\leq \frac{e}{e_P}|P|+|P|+|U|\\
\leq & \frac{e}{e_P}|P|+v(F)\leq \frac{e}{2}+e(F)+\Delta(F)\leq e+\Delta(F),
\end{align*}
since $e(F) \le e/2$. This completes the proof of the lemma.
\end{proof}

Now we have all the necessary ingredients to prove Theorem~\ref{thm:power saving}.

\begin{proof}[Proof of Theorem~\ref{thm:power saving}.]
Throughout the proof, we will assume that we have a linear hypergraph $\cH$ with $\Omega(n^{2-\eps})$ edges, where $\eps<e^{-5}$ and $e$ is the required number of edges in the final configuration. Note that the assumption that $\cH$ is linear is without any loss of generality, since if any of the pairs $(x, y)\in V(\cH)^2$ was in at least $e$ edges then we would have an $(e+2, e)$-configuration. Hence, one can keep a positive constant (dependent only on $e$) fraction of all edges in $\cH$ while ensuring that no pair $(x, y)\in V(\cH)^2$ appears in more than one edge.

In case $e$ is small, say $e<512$, we can use Lemma~\ref{lemma:start} to find many copies of $K_{\lceil \sqrt{e}\rceil, \lceil \sqrt{e}\rceil}^{+}$ in $\cH$. The deficiency of this hypergraph is $2\lceil\sqrt{e}\rceil\leq \lfloor \log_2 e\rfloor +38$ as long as $e<512$, which is sufficient for the proof of our theorem in this case.

For $e\geq 512$, we apply Lemma~\ref{lemma:start} with $s=t=16$ to find an eligible hypergraph $F_0$ such that $\cH$ contains $\Omega(n^{k_0-s^3 t \eps})$ copies of $F_0$, where $k_0=\Delta(F_0) = 32$. Note that $e(F_0) = 256$. Let us now fix the parameters $r=e$ and $\ell=\lfloor \log_2 e/e(F_0)\rfloor-1$ (where $\ell\geq 0$ since $e\geq 512$). Observe that the parameters were chosen so that $2^\ell s^3t \eps\leq 1/2$. Having fixed these parameters, we will show how to find a $(v, e)$-configuration in $\cH$ with $v-e\leq \lfloor \log_2 e\rfloor +26$.

The idea is to iteratively apply Lemma~\ref{lemma:iteration step} with $m=2$ to construct a sequence of eligible hypergraphs $F_0, F_1, F_2, \dots, F_\ell$ with the property that either $\cH$ contains an $(r, F_j)$-sunflower for some $0\leq j\leq \ell$ or $\cH$ contains $\Omega(n^{k_0+j-2^j \eps'})$ copies of $F_j$ for each hypergraph $F_j$ in this sequence, where $\eps'=s^3t\eps$. Indeed, the construction of this sequence is given by Lemma~\ref{lemma:iteration step}, where we set $F_{j+1}=F_j'$ for every $0 \le j \le \ell-1$. Since $e(F_j')=2e(F_j)$ and $\Delta(F_j')=\Delta(F_j)+1$, this construction satisfies $e(F_j)=2^je(F_0)$ and $\Delta(F_j)=\Delta(F_0)+j$ for all $0 \le j \le \ell$.
If we assume that $\cH$ does not contain an $(r, F_j)$-sunflower for any $0\leq j\leq \ell$, one can show by induction on $j$ that $\cH$ contains $\Omega(n^{k_0+j-2^j \eps'})$ copies of $F_j$ for each $0\leq j\leq \ell$. The base case is given by Lemma~\ref{lemma:start}, since $\cH$ contains $\Omega(n^{k_0-\eps'})$ copies of $F_0$. For $j\ge 1$, from the assumption that $\cH$ contains $\Omega(n^{k_0+j-1-2^{j-1}\eps'})$ copies of $F_{j-1}$ and that $\cH$ contains no $(r, F_{j-1})$-sunflower, Lemma~\ref{lemma:iteration step} allows us to deduce that $\cH$ contains $\Omega(n^{k_0+j-2^j \eps'})$ copies of $F_j$, allowing us to continue the induction. Recall that the parameters were set such that $2^j \eps'\leq 2^\ell s^3 t\eps \leq \frac{1}{2}$ and hence Lemma~\ref{lemma:iteration step} can be applied for all $j\leq \ell$.

If $\cH$ contains an $(r, F_j)$-sunflower for some $0\leq j\leq \ell$, then Lemma~\ref{lemma:sunflower} shows that $\cH$ contains an $(e+\Delta(F_j), e)$-configuration. Indeed, the assumptions of Lemma~\ref{lemma:sunflower} are satisfied since $e(F_j)\leq 2^\ell e(F_0)\leq \frac{e}{2e(F_0)}e(F_0)=e/2$ and $r=e$ (which implies that the sunflower has at least $e$ edges, since $F_j$ has no isolated vertex). Moreover, note that finding an $(e+\Delta(F_j), e)$-configuration in $\cH$ would indeed suffice to prove Theorem~\ref{thm:power saving} since
\[\Delta(F_j)=\Delta(F_0)+j\leq 32+\ell\leq 32+\lfloor \log_2 e/e(F_0) \rfloor -1= \lfloor \log_2 e\rfloor +23.\]

The only case left to consider is when $\cH$ does not contain an $(r, F_j)$-sunflower for any $0\leq j\leq \ell$. Applying Lemma~\ref{lemma:iteration step} to $F_\ell$ with $m=4$ and using the fact we have no $(r, F_\ell)$-sunflower, we arrive at the conclusion that $\cH$ contains a hypergraph $F_\ell'$ obtained by gluing $4$ copies of $F_\ell$ along a good set $A\subset V(F_\ell)$ of size $\Delta(F_\ell)-1$. The hypergraph $F_\ell'$ satisfies $\Delta(F_\ell')=\Delta(F_\ell)+3$ and $e(F_\ell')=4e(F_\ell)\in [e, 2e]$. Since $F_\ell'$ is formed by gluing $4$ copies of $F_\ell$ along a set of size $\Delta(F_\ell)-1$ and $F_\ell$ is eligible, the number of vertices of degree more than $1$ in $F_\ell'$ is at most $4(e(F_\ell)/4-\Delta(F_\ell))+\Delta(F_\ell)-1\leq e(F_\ell')/4$. Hence, $F_\ell'$ contains at least $v(F_\ell')-e(F_\ell')/4 = e(F_\ell')+\Delta(F_\ell')-e(F_\ell')/4 \ge e(F_\ell')/2$ vertices of degree $1$, so one can remove $e(F_\ell')-e\leq e(F_\ell')/2$ vertices of degree $1$ from $F_\ell'$ without changing the deficiency of the configuration. In this way, for any $e \ge 512$, one also obtains an $(e+\Delta(F_\ell'), e)$-configuration with \[ \Delta(F_\ell')\leq \ell+3+\Delta(F_0)\leq 32+\lfloor \log_2 e/e(F_0)\rfloor+2= \lfloor \log_2 e\rfloor +26.\]
This completes the proof of Theorem~\ref{thm:power saving}.
\end{proof}

\section{Concluding remarks}\label{sec:concluding remarks}

In order to improve the coefficient of $\log_2 e$ in Theorem \ref{thm:power saving} (or to replace it with $o(\log_2 e)$), it seems that new ideas are needed. Our strategy was to find many copies of each of $F_0,F_1,\dots,F_{\ell}$, where $F_{i+1}$ is obtained by gluing two copies of $F_i$ along an independent set of size $\Delta(F_i)-1$. Another approach, whose natural limit also seems to be deficiency $\log_2 e$, is as follows. Similarly to our approach in this paper, we define a sequence $H_0,H_1,\dots,H_{\ell}$ of hypergraphs such that each of them has relatively low deficiency and we can find many copies of each of them in our host hypergraph. Also, $H_{i+1}$ is built from $H_i$ recursively, but differently from the way we built $F_{i+1}$. To be more precise, we fix an independent set $I$ of size $k:=\Delta(H_i)$ in $H_i$ and a complete $k$-partite $k$-uniform hypergraph $K$, and we attach a copy of $H_i$ to $K$ for each $e\in E(K)$ by identifying $I$ with $e$ (in a way that all other vertices in the copies of $H_i$ are distinct). Let us explain why we may expect to find many copies of these hypergraphs in our host hypergraph. Assume that we have already shown that the number of copies of $H_i$ in our host hypergraph is at least roughly $n^{\Delta(H_i)-e(H_i)\eps}=n^{k-e(H_i)\eps}$ (which is the expected number of copies in a random $3$-uniform hypergraph with $n^{2-\eps}$ edges). Define an auxiliary $k$-uniform hypergraph $\mathcal{K}$ in which hyperedges are the $k$-sets corresponding to $I$ in the copies of $H_i$ in our host graph. We expect at least $n^{k-e(H_i)\eps}$ edges in this hypergraph, so if $\eps$ is sufficiently small, then we indeed expect to find many copies of $\mathcal{K}$ in this auxiliary hypergraph, giving many (homomorphic) copies of $H_{i+1}$. This also shows why we cannot take $\mathcal{K}$ to be an arbitrary $k$-uniform hypergraph: this number of edges is only enough to find $k$-partite subgraphs. It is not hard to see that this approach will not lead to deficiency fewer than $\log_2 e$ for configurations with $e$ edges.

\section*{Acknowledgments} 
We are grateful to David Conlon and Lior Gishboliner for helpful comments.

\bibliographystyle{amsplain}


\begin{dajauthors}
\begin{authorinfo}[oliver]
  Oliver Janzer\\
  EPFL\\
  Lausanne, Switzerland\\
  oliver.janzer\imageat{}epfl\imagedot{}ch \\
  \url{https://people.epfl.ch/oliver.janzer}
\end{authorinfo}
\begin{authorinfo}[abhishek]
  Abhishek Methuku\\
  UIUC\\
  Urbana-Champaign, United States\\
  methuku\imageat{}illinois\imagedot{}edu \\
  \url{https://math.illinois.edu/directory/profile/methuku}
\end{authorinfo}
\begin{authorinfo}[aleksa]
  Aleksa Milojevi\'c\\
  ETH Z\"urich\\
  Z\"urich, Switzerland\\
  aleksa.milojevic\imageat{}math\imagedot{}ethz\imagedot{}ch\\
  \url{https://people.math.ethz.ch/~amilojevi}
\end{authorinfo}
\begin{authorinfo}[benny]
  Benny Sudakov\\
  ETH Z\"urich\\
  Z\"urich, Switzerland\\
  benjamin.sudakov\imageat{}math\imagedot{}ethz\imagedot{}ch\\
  \url{https://people.math.ethz.ch/~sudakovb}
\end{authorinfo}
\end{dajauthors}

\end{document}